\documentclass[twoside,11pt]{article}

\usepackage[preprint]{jmlr2e}
\usepackage{amsmath}
\usepackage[nameinlink]{cleveref}
\usepackage{comment}
\usepackage{algorithm}
\usepackage[noend]{algpseudocode}

\usepackage{caption}
\usepackage{subcaption}
\usepackage{makecell}
\usepackage{booktabs}

\newcommand{\weak}{\rightharpoonup}

\newcommand{\be}{\begin{equation}}
\newcommand{\ee}{\end{equation}}
\newcommand{\ba}{\begin{eqnarray}}
\newcommand{\ea}{\end{eqnarray}}
\newcommand{\beq}{\begin{equation}}
\newcommand{\eeq}{\end{equation}}

\newcommand{\plotErrors}[4]{
\begin{figure}[h!]
         \centering
\begin{subfigure}[b]{0.3\textwidth}
\centering
         \includegraphics[width=\textwidth,height=5cm]{"Figuras/"#1"/L2_err_norm_control_"#2.png}
         \caption{{\expandafter\MakeUppercase#2} MNSE$_u$}
         \label{#4:#2mnseu}
     \end{subfigure}
     \hfill
\begin{subfigure}[b]{0.3\textwidth}
\centering
         \includegraphics[width=\textwidth,height=5cm]{"Figuras/"#1"/L2_err_norm_state_"#2.png}
         \caption{{\expandafter\MakeUppercase#2} MNSE$_y$}
         \label{#4:#2mnsey}
     \end{subfigure}
     \hfill
\begin{subfigure}[b]{0.3\textwidth}
\centering         
         \includegraphics[width=\textwidth,height=5cm]{"Figuras/"#1"/rel_err_"#2.png}
         \caption{{\expandafter\MakeUppercase#2} MNAE$_J$}
         \label{#4:#2mnae}
     \end{subfigure}
\caption{{\expandafter\MakeUppercase#2} errors #3.}
\label{#4:#2error}
\end{figure}}

\renewcommand{\leq}{\leqslant}

\renewcommand{\geq}{\geqslant}
\renewcommand{\ge}{\geqslant}
\providecommand{\norm}[1]{\left\lVert#1\right\rVert}

\def \R {\mathbb{R}}

\def \B {\mathcal{B}}
\def \N {\mathbb{N}}

\def \dis {\displaystyle}

\def\beq{\begin{equation}}
\def\eeq{\end{equation}}
\def\ecart{\noalign{\medskip}}
\def\ba{\begin{array}}
\def\ea{\end{array}}

\newtheorem{theo}{Theorem}
\newtheorem{prop}{Proposition}
\newtheorem{defi}{Definition}

\newtheorem{rem}[defi]{Remark}
\newtheorem{hypo}{Hypothesis}[section]
\DeclareMathOperator*{\argmin}{argmin}

\jmlrheading{1}{2000}{1-48}{4/00}{10/00}{meila00a}{Karl Kunish and Donato Vásquez-Varas}

\ShortHeadings{Optimal polynomial feedback laws for finite horizon control problems}{Kunisch and Vásquez-Varas}
\firstpageno{1}

\begin{document}

\title{Optimal polynomial feedback laws for finite horizon control problems}

\author{\name Karl Kunisch \email karl.kunisch@uni-graz.at \\
       \addr Radon Institute for Computational and Applied Mathematics\\
Austrian Academy of Sciences\\
and\\
Institute of Mathematics and Scientific Computing\\
University of Graz\\
Heinrichstraße 36, A-8010 Graz, Austria
       \AND
       \name Donato Vásquez-Varas \email donato.vasquez-varas@ricam.oeaw.ac.at\\ \addr Radon Institute for Computational and Applied Mathematics\\
Austrian Academy of Sciences\\
Altenbergerstraße 69, A-4040 Linz, Austria
}

\editor{}

\maketitle

\begin{abstract}
A  learning technique for finite horizon optimal control problems and its approximation based on polynomials is
analyzed. It allows to circumvent, in part, the curse dimensionality which is involved
when the feedback law is constructed by using the Hamilton-Jacobi-Bellman (HJB) equation. The convergence of the method is analyzed, while paying special attention to avoid the use of a global Lipschitz condition on the nonlinearity which describes the control system. The practicality and efficiency of the method is illustrated
by several examples. For two of them a direct approach based on the HJB equation would be unfeasible.
\end{abstract}


\section{Introduction}
The synthesis of optimal feedback laws for non-linear control problems remains to be a difficult task. The main approach for obtaining feedback laws relies on the  dynamical programming and it involves solving a Hamilton-Jacobi-Bellman (HJB) equation. It is well known that this equation suffers from the curse of dimensionality, which makes a direct approach infeasible. In the last years many attempts were made in order to alleviate this obstacle. Without intending to be exhaustive, here we present some works related to this topic: representation formulas \cite{Chow1,Chow2,Chow3,DarbonOsher}, approximating the HJB equation by neuronal networks \cite{Han,Darbon,Nusken,Onken,Ito,KW1,KW2,Ruthotto}, data driven approaches \cite{Nakamura1,Nakamura2,AzKaKK,Kang,Albi,Dolgov2},  max-plus algebra methods \cite{Akian,Gaubert,Dower}, polynomial approximation \cite{Kalise1,Kalise2}, tensor decomposition
methods \cite{Horowitz,Stefansson,Gorodetsky,Dolgov,Oster,Oster2}, POD methods \cite{Alla2,KunischVolk}, tree structure algorithms \cite{Alla1}, and sparse grids  techniques \cite{BokanowskiGarckeGriebelPo, Garcke, KangWilcox}, see also the proceedings volume \cite{KaliseKuRa}. Among the classical methods for solving the HJB equation we mention finite difference schemes \cite{Bonnans}, semi-Lagrangian schemes \cite{Falcone}, and policy iteration \cite{Alla, Beard, Puterman, Santos}.
\par
 In the present  work we analyze a learning approach to obtain optimal feedbacks laws for finite horizon, non-linear optimal control problems. This is an extension of the developments made for infinite horizon control problems in \cite{KuVaWal} and \cite{KW1}, and for finite horizon problems in \cite{KW2}.  While in \cite{KW1} and \cite{KW2} we focus on approximations and numerical realization  by neural networks, in \cite{KuVaWal} we started to use monomials as approximating ansatz functions for the feedback operators to be learned. In the analysis part of this paper the optimality conditions are analysed and a relation  between them and the HJB equation is established, which allows to compare  the value function of the control problem and the one which is obtained by the learning approach. Compared to \cite{KW2} we choose a simpler cost functional here, which 
  only involves the value function itself, wereas in  \cite{KW2} we used also first and second order derivatives. 
  In the numerical part, we demonstrate the efficiency of the approach by diverse applications. Let us mention that we also tested approximation of the time-dependent component of the value
 feedback function by Legendre polynomials rather than polynomials. In numerical practice the latter outperformed the former. For this reason we focus here on the approximation by monomials in space as well as in time.   
\par
The proposed learning technique results in  an infinite dimensional optimization  problem. For practical realization  it must be discretized. For this purpose, we specify a rather general framework for which convergence  of  discretization schemes can be established.  Due to their simplicity and the encouraging  results obtained in \cite{KuVaWal}, we  use polynomials as ansatz functions for our concrete numerical realization. In this  case the general assumptions for  convergence hold true, provided  that the value function is sufficiently smooth. We point out that differently from \cite{KuVaWal}, the technique used here to prove  convergence of the method utilizes that  the value function solves the HJB equation. This, together with the  fact that we consider a finite time horizon allows us to get a better convergence order than the one obtained in \cite{KuVaWal}.
\par
The structure of this work is as follows. In \Cref{ControlTheory} we give a brief review of dynamical programming theory. In \Cref{StatementLearningProblem} the formulation  of the learning problem is introduced. In \Cref{OptimalityConditions} we present  and analyse the optimality conditions of the learning problem. In \Cref{FiniteDimensionalLearningProblem}  a finite dimensional approximation of the learning problem together with a convergence result are presented. The learning algorithm is described in \Cref{OptimizationAlg}, and the specific choice of polynomials as ansatz functions for the feedback operator in Section \ref{PolynomialLearningProblem}. Finally, in \Cref{NumericalExperimentsSection} we present three numerical experiments which show that our algorithm is able to solve non-linear and high (here the dimension is 40) dimensional control problems in a standard laptop environment.

\par We end the section by introducing some notation which is needed in the following. For $T>0$, $d$ a non-negative integer, and $A\subset\R^{d}$, the space $C([0,T];A)$ denotes the space of continuous functions from $[0,T]$ to $A$, and for $p\in[1,\infty)$, $L^{p}((0,T);A)$ stands for  the space of  $p$ integrable function from $(0,T)$ to $A$. For a compact set $K\subset\R^{d}$ and a non-negative integer $k$, the space $C^{k}(K)$ denotes the space of $k$-times continuously differentiable functions from $K$ to $\R$, and $C([0,T];C^{k}(K))$ stands for  the space of functions from $[0,T]\times K$ to $\R$ that are continuous in their first  and $k$ times continuously differentiable in their second variable. We utilize the following norm for this space:
$$ \norm{v}_{C([0,T];C^{k}(K))}=\max_{t\in [0,T]}\norm{v(t,\cdot)}_{C^{k}(K)},  $$
where $\norm{\cdot}_{C^{k}(K)}$ is the usual norm on $C^{k}(K)$. {For $y\in\R^{d}$ we denote the $p-$norm  by $|y|_{p}$ for $p\in (1,\infty)\setminus\{2\}$, the supremum norm by $|y|_{\infty}$, and the Euclidean norm simply by $|y|$.}

\section{Control theory and dynamical programming}
\label{ControlTheory}
We aim for an optimal feedback law for the following finite horizon  problem
\beq
 (P) \quad\min_{u} J(t_0,y_0,u):=\int_{t_0}^{T}\left(\ell (t,y)+\frac{\beta}{2}|u|^{2}\right) dt+g(y(T))\label{ControlProblem}\eeq
\beq s.t.\quad y'(t)=f(t,y(t))+Bu(t),\quad y(t_0)=y_0  \label{OpenLoop}
\eeq
where $T>0$, $t_0\in(0,T)$,  $y_0\in \Omega$, $\beta>0$, $u\in L^2((t_0,T); \R^M)$, $f\in C([0,T];C^{1}(\overline{\Omega};\R^{d}))$,  $B\in \R^{d\times M}$ , and  $\ell\in C([0,T]; C^{1}(\overline{\Omega}))$ and $g\in C^{1}(\overline{\Omega})$ are bounded from below by 0. Further
 $\Omega$ is an open, convex, and bounded subset of $\R^{d}$.
For this problem, we define the associated value function by
\beq V(t_0,y_0)=\min_{u} J(t_0,y_0,u).\eeq
If the value function $V$ is in $C^{1}([0,T]\times\Omega)$, then it satisfies the Hamilton-Jacobi-Bellman equation
\beq
 \frac{dV}{dt}(t,y)+\ell(y)+\min_{u\in\R^{M}}\left\{ \nabla_{y} V(t,y)^{\top}(f(y)+Bu)+\frac{\beta}{2}|u|^{2}\right\}=0\mbox{ in }[0,T]\times\Omega, \label{HJBeq}\eeq
\beq \quad V(T,y)=g(y)\mbox{ in }\Omega. \label{HJBbc} \eeq
Moreover, if for an initial condition $y_0\in \Omega$ the optimal solution of \eqref{ControlProblem} is such that the respective state $y^{*}$satisfies $y^{*}(t)\in\Omega$ for all $t\in[t_0,T]$, then we can define a feedback law by evaluating

\beq
{{u^{*}(t,y)=-\frac{1}{\beta}B^{\top}\nabla_y V(t,y)}}\label{OptimalFeedbackEq}
\eeq
along $y=y^*(t)$,
and $y^{*}$ solves the closed loop problem
\beq \frac{d}{dt}y^{*}(t)=f(y^{*})+Bu^{*}(t,y^{*}(t)),\quad y^{*}(t_0)=y_0. \label{closedloop1}\eeq
This implies that once we solve the HJB equation we can obtain an optimal feedback law by using \eqref{OptimalFeedbackEq}. However, it is well known that the HJB equation suffers from the curse of dimensionality, which makes this approach infeasible for large  $d$. Instead we propose a learning approach that we will describe in the following section.
\section{Statement of the learning problem}
\label{StatementLearningProblem}
To start we introduce the relevant  hypotheses and definitions. Let  $\omega\subset\Omega$ be a open subset of $\Omega$ such that $\overline{\omega}\subset\Omega$. For $v\in C([0,T];C^{2}(\overline{\Omega}))$, we define $U_{v}(t,y)$ and $F_v(t,y)$ by
$$ U_v(t,y)=-\frac{1}{\beta}B^{\top}\nabla_y v(t,y)\mbox{ and }F_v(t,y)=f(t,y)+BU_v(t,y) \mbox{ for all }(t,y)\in [0,T]\times \overline{\Omega}.$$
We note that $U_v(t,y)$ satisfies
\beq
U_v(t,y)= \argmin_{U\in\R^{m}}\left\{\frac{\beta}{2}|u|^{2}+\nabla_{y} v(t,y)^{\top}(f(t,y)+BU)\right\} \mbox{ for all }(t,y)\in [0,T]\times\overline{\Omega}.
\eeq
Clearly, $U_v$ is a $C^{1}$ function with respect to $y$. Combining this with the fact that $\overline{\Omega}$ is a convex and compact subset of $\R^{d}$, we have that $U_v$ and $F_v$ are Lipschitz functions in their second argument, namely, defining
$$ L_{U_v}=\frac{|B|}{\beta}\norm{\nabla_y^2 v}_{C([0,T]\times\overline{\Omega};\R^{d\times d})},\quad L_{F_v}=L_f+|B|L_{U_v}, $$
we have
\beq |U_v(t,y_1)-U_v(t,y_2)|\leq L_{U_v}|y_1-y_2|,\quad |F_v(t,y_1)-F_v(t,y_2)|\leq L_{F_v}|y_1-y_2|,\label{LipchitzCondFU}\eeq
 for all $t\in [0,T]$ and $y_1,y_2\in \overline{\Omega}$. In the following we assume that $v$ is such that for all $(t_0,y_0)\in [0,T]\times\omega$, there exists a unique $y\in C^{1}([0,T];\R^{d})$ satisfying
\beq y'(t)=F_{v}(t,y(t)),\quad \mbox{for all } t\in [t_0,T],\quad y(t_0)=y_0.\label{closedloop2}\eeq
 We extend $y$ as follows
\beq  \tilde{y}(t)=\left\{\begin{array}{ll}
    \dis y(t)  & \mbox{ if }t\in [t_0,T]  \\
    \ecart\dis y_0+\int_{t_0}^{t}F_v(s,y_0)ds & \mbox{ if }t\in [0,t_0]
\end{array}\right. \label{closedloop2extension} \eeq
which is a continuously differentiable function in $[0,T]$. Associated  to $v$, we define
\beq
\mathcal{G}(v):=\{(t,\tilde{y}(t;t_0,y_0)):\ (t_0,y_0)\in (0,T)\times\omega, \ t\in [t_0,T) \}\subset(0,T)\times\R^{d},
\eeq
\beq
 \mathcal{R}(v):=\{y\in\R^{d}:\ \exists t\in (0,T),\ (t,y)\in \mathcal{G}(v)\}\subset\R^{d}, \eeq
where $\tilde{y}(t;t_0,y_0)$ is defined in \eqref{closedloop2extension}. Due to the fact that for $t_0\in (0,T)$ and $t\in [t_0,T]$ the function $y_0\in\omega\subset\R^{d}\mapsto \tilde{y}(t;t_0,y_0)$ onto its range in $\R^{d}$ is a diffeomorphism, we know that $\mathcal{G}(v)$ and $\mathcal{R}(v)$ are open sets.\par
Additionally, we define $\mathcal{D}$ as a subset of $C([0,T],C^{2}(\overline{\Omega}))$ of functions such that problem \eqref{closedloop2} has a solution and the closure of the union of all the trajectories with $(t_0,y_0)\in [0,T]\times\omega$ is contained in  $\Omega$, i.e.,
$$
\mathcal{D}=\left\{v\in C([0,T],C^{2}(\overline{\Omega})):\
\begin{array}{l}
    \dis \mbox{ problem }\eqref{closedloop2} \mbox{ has a solution }\forall (t_0,y_0)\in [0,T]\times\omega\\
    \ecart\dis\mbox{ and } \overline{\mathcal{R}(v)}\subset \Omega
\end{array}
\right\}.
$$
\par
We define $\Phi$ as the solution mapping of \eqref{closedloop2extension}, that is,
\beq \begin{array}{c}
    \Phi:[0,T]\times \omega\times \mathcal{D}\mapsto C([0,T];\overline{\Omega}) \\
     (t_0,y_0,v)\in [0,T]\times \omega\times \mathcal{D}\mapsto\Phi(t_0,y_0,v)= \tilde{y}\in C([0,T];\overline{\Omega})
\end{array}  \eeq
 where $\tilde{y}$ is the function given by \eqref{closedloop2extension} {{and the domain is endowed with the norm of $\R\times \R^d \times C([0,T],C^{2}(\overline{\Omega}))$.}}
\par
We formulate the following learning problem
\beq
\dis\min_{v\in \overline{\mathcal{D}}}\mathcal{J}(v):=\frac{1}{T|\omega|}\int_{\omega}\int_{0}^{T}\mathcal{V}(t_0,y_0,v)dy_0dt_0 \label{LearningProblem}\eeq
where $\mathcal{V}:[0,T]\times\omega\times\mathcal{D}$ is defined by
\beq \mathcal{V}(t_0,y_0,v):=\int_{t_0}^{T}\left(\ell(t,y(t))+\frac{\beta}{2}|U_v (t,y(t))|^{2}\right) dt+g(y(T)), \eeq
with $y(t)=\Phi(t_0,y_0,v)(t)$. We shall prove that the function $\Phi$, $\mathcal{V}$, and $\mathcal{J}$ are continuously differentiable functions.\par

\begin{rem}
 We shall verify that if $V\in \mathcal{D}$, then there exists a non-trivial neighborhood of $V$ in $C([0,T];C^{1}(\overline{\Omega}))$ which is contained in $\mathcal{D}$. Moreover, $\mathcal{V}(t_0,y_0,v)\geq V(t_0,y_0)$ for all $v\in C([0,T];C^{1}(\overline{\Omega}))$, which implies that the value function is a solution for the learning problem. Further, if  $v^{*}$ is an optimal solution of \eqref{LearningProblem}, then
 $$ \int_{\Omega}\int_{0}^{T}|\mathcal{V}(t_0,y_0,v^{*})-V(t_0,y_0)| dt_0 dy_0=\int_{\Omega}\int_{0}^{T}\left(\mathcal{V}(t_0,y_0,v^{*})-V(t_0,y_0) \right)dt_0 dy_0\leq 0$$
 and by the continuity of $\mathcal{V}$ and $V$, we obtain that $\mathcal{V}(t_0,y_0,v^{*})=V(t_0,y_0)$ for all $(t_0,y_0)\in [0,T]\times\omega$. Hence, by the dynamical programming principle
 the feedback $U_{v^{*}}$ is optimal in $\mathcal{G}(v^{*})$.
\end{rem}

\section{Optimality conditions}
\label{OptimalityConditions}
Now, we study the necessary and sufficient optimality conditions for the learning problem \eqref{LearningProblem}. To this end, we first analyze the smoothness of the functions $\Phi$, $\mathcal{V}$ and $\mathcal{J}$.
\begin{lemma}
\label{lemma:OpenD}
Let $v\in \mathcal{D}$. Then, there exists $\delta>0$ such that if for $\tilde{v}\in C([0,T];C^{2}(\overline{\Omega}))$ it holds $\norm{\tilde{v}-v}_{C([0,T],C^{1}(\overline{\Omega}))}<\delta$, then $\tilde{v}\in \mathcal{D}$. Further, $\mathcal{D}$ is an open subset of $C([0,T];C^{2}(\overline{\Omega}))$ and
\beq \norm{\Phi(t_0,y_0,v)-\Phi(t_0,y_0,\tilde{v})}_{C([0,T];\R^{d})}\leq C\norm{\tilde{v}-v}_{C([0,T];C^{1}(\overline{\Omega}))},\ \forall (t_0,y_0,\tilde{v})\in [0,T]\times \omega\times \mathcal{D}\label{lemma:OpenD:contPhiv}\eeq
\end{lemma}
where $C$ depends on $B,f,T$ and $v$.
\begin{proof}
 Denote by $y$ the solution of \eqref{closedloop2} for $v$. By \eqref{LipchitzCondFU} and the local existence Theorem, there exists an existence time $\bar{T}>t_0$ and a function $\tilde{y}\in C([0,\tilde{T}];\R^{d})$ such that
 $$ \tilde{y}'(t)=F_{\tilde{v}}(t,\tilde{y}(t))\mbox{ for all }t\in (t_0,\tilde{T})\mbox{ and }\tilde{y}(t_0)=y_0. $$
We assume that $\tilde{T}$ is the maximum time that satisfies $\tilde{y}(t)\in \Omega$ for all $t\in (t_0,\tilde{T})$. We shall prove that $\tilde{T}\geq T$ by contradiction. Assume that $\tilde{T}<T$. This implies that $\tilde{y}(\tilde{T})\in\partial\Omega$, otherwise $\tilde{y}(\tilde{T})\in\Omega$ and we could use the existence and uniqueness Theorem to extend $\tilde{y}$ to some $\bar{T}>\tilde{T}$ with $\tilde{y}(t)\in\Omega$ for all $t\in (t_0,\bar{T})$, which is a contradiction. Set $z=y-\tilde{y}$, then for $t\in (0,\tilde{T})$ we have
\begin{equation} \begin{array}{ll}
    \dis \frac{1}{2}\frac{d}{dt}|z(t)|^{2} &\dis= (F_{v}(t,y(t))-F_{\tilde{v}}(t,\tilde{y}(t)))z(t) \\
    \ecart \dis & \dis\leq (F_v(t,y(t))-F_v(t,\tilde{y}(t))+F_v(t,\tilde{y}(t))-F_{\tilde{v}}(t,\tilde{y}(t)))z(t)\\
    \ecart\dis & \dis\leq  L_{F_v}|z(t)|^{2}+\frac{1}{\beta}|B|^{2}|z(t)|\norm{\tilde{v}-v}_{C([0,T];C^{1}(\overline{\Omega}))},
\end{array}
\end{equation}
By a generalized Gronwall inequality (see Theorem 21 in \cite{Dragomir}) this estimates implies
$$ |y(t)-\tilde{y}(t)|=|z(t)| \leq\frac{|B|^{2}}{\beta L_{F_v}}(\exp(L_{F_v}(T-t_0))-1)\norm{\tilde{v}-v}_{C([0,T];C^{1}(\overline{\Omega}))}$$
for all $t\in (t_0,\tilde{T})$. In particular, if $\tilde{v}\in\mathcal{D}$, then we have $\tilde{T}\geq T$ and the last inequality implies \eqref{lemma:OpenD:contPhiv}. By the definition of $\mathcal{D}$ there exists $\varepsilon>0$ such that
$$\mathcal{R}_{\varepsilon}=\{x\in \R^{d}: dist(x,\overline{\mathcal{R}(v)})<\varepsilon\}\subset\Omega $$
and we can choose $\delta$ such that $|z(t)|\leq \frac{\varepsilon}{2}$ for all $t\in [t_{0},\tilde{T}]$. By the continuity of $\tilde{y}$, this implies that $\tilde{y}(\tilde{T})\in\mathcal{R}_\varepsilon\in\Omega$, which is a contradiction and therefore $\tilde{T}\geq T$. Moreover, since this holds uniformly in $(t_0,y_0)\in (0,T)\times\omega$, we have that $\mathcal{R}(\tilde{v})\subset\mathcal{R}_{\varepsilon}\subset\Omega$. Since $\mathcal{R}_{\varepsilon}$ is close, we have  $\overline{\mathcal{R}(\tilde{v})}\subset\Omega$, which implies that $\tilde{v}\in D$ and thus $\mathcal{D}$ is open.
\end{proof}

\begin{prop}\label{prop:DifPhi}
 The functional $\Phi$ is locally Lipschitz in $[0,T]\times\omega\times\mathcal{D}$ and continuously differentiable. Further, setting  $z_1=D_{t_0}\Phi(t_0,y_0,v)\in C([0,T];\R^{d})$, $z_2=D_{y_0}\Phi(t_0,y_0,v)\in C([0,T];\R^{d\times d})$, and $z_{3}= D_{v}\Phi(t_0,y_0,v)\cdot\phi \in C([0,T];\R^{d})$ for $\phi\in C([0,T],C^{2}(\overline{\Omega}))$, $v\in\mathcal{D}$ and $(t_0,y_0)\in (0,T)\times\omega$,  we have
\begin{flalign}
 z_1(t)&=\left\{\begin{array}{ll}\dis -F_{v}(t_0,y_0)+\int_{t_0}^{t} D_{y}F_v(s,y(s))z_1(s) ds & \mbox{for }t\in (t_0,T],  \\
\ecart \dis  -F_{v}(t_0,y_0)  & \mbox{for }t\in [0,t_0], \end{array}\right. \label{PhiPartial:t0}\\
z_2(t)&=\left\{\begin{array}{ll}\dis I+\int_{t_0}^{t} D_{y}F_v(s,y(s))z_2(s) ds & \mbox{for }t\in (t_0,T],  \\
\ecart \dis  I+\int_{t_0}^{t} D_{y}F_v(s,y_0)z_2(s)ds & \mbox{for }t\in [0,t_0], \end{array}\right. \label{PhiPartial:y0}\\
z_3(t)&=\left\{\begin{array}{ll}\dis \int_{t_0}^{t}\left(D_{y}F_v(s,y(s))z_3(s)+BU_\phi(s,y(s))\right) ds & \mbox{for }t\in (t_0,T],  \\
\ecart \dis \int_{t_0}^{t}B U_\phi(s,y_0) ds& \mbox{for }t\in [0,t_0], \end{array}\right.\label{PhiPartial:v} \end{flalign}
where $y=\Phi(t_0,y_0,v)$.
\end{prop}
\begin{proof}
By \Cref{lemma:OpenD} we already know that $\Phi$ is locally Lipchitz with respect to $v$. Additionally, the proof of the Lipschitz continuity of $\Phi$ with respect to $y_0$ is analogous to the proof of Proposition 1.10.1 in \cite[Chapter~2]{cartan}. Similarly, the differentiability with respect to $y_0$ and $v$ are analogous to the proof of Theorems 3.4.2 and 3.6.1 in \cite[Chapter~2]{cartan}. The proof of the Lipschitz continuity and the differentiability of $\Phi$ with respect to $t_0$ is similar to the proofs of Proposition 1.10.1 and Theorem 3.4.2 in \cite[Chapter~2]{cartan}, but more delicate. For this reason, herein we provide the proof of the Lipschitz continuity and the differentiability of $\Phi$ with respect to $t_0$. \par
To verify the differentiability of $\Phi$ with respect to $t_0\in (0,T)$ we utilize the implicit function theorem. For fixed $(y_0,v)\in \Omega\times\mathcal{D}$ we define the mapping $\Psi: (0,T)\times C([0,T];\R^{d})\mapsto C([0,T];\R^{d})$ by
\begin{equation}
\Psi(t_0,y)(t)=\left\{\begin{array}{ll}
\dis y(t)-\int_{t_0}^{t}F_v(s,y(s))ds-y_0 & \mbox{for }t\in [t_0,T],\\
\ecart\dis y(t)-\int_{t_0}^{t}F_v(s,y_0)ds-y_0 & \mbox{for }t\in [0,t_0].
\end{array}\right.
\end{equation}
Its partial derivatives are given by
\begin{equation}
(D_y\Psi(t_0,y) \!\cdot \! z)(t)=\left\{\begin{array}{ll}
\dis z(t)-\int_{t_0}^{t}D_yF_v(s,y(s))z(s)ds & \mbox{for }t\in [t_0,T],\\
\ecart \dis z(t) & \mbox{for }t\in [0,t_0],
\end{array}\right.
\end{equation}
for $z\in C([0,T];\R^{d})$, and
\begin{equation}
\frac{d}{dt_0}\Psi(t_0,y)(t)=\left\{\begin{array}{ll}
\dis F_v(t_0,y(t_0)) & \mbox{for }t\in [t_0,T],\\
\ecart \dis F_v(t_0,y_0) & \mbox{for }t\in [0,t_0].
\end{array}\right.
\end{equation}
Thus $\Psi$ is $C^{1}$ with respect to $y$ and it is $C^1$ with respect to $t_0$ at $(t_0, \Phi(t_0,y_0,v))$. Moreover, $\Psi(t_0,\Phi(t_0,y_0,v))=0$
and $D_y\Psi(t_0,\Phi(t_0,y_0,v))$ is surjective. Indeed for arbitrary $h\in C([0,T];\R^{d})$, there exists a unique $\tilde{z}\in C([t_0,T];\R^{d})$ satisfying
$$\tilde{z}(t)=\int_{t_0}^{t}D_yF_v(s,\Phi(t_0,y_0,v))\tilde{z}(s)ds+h(t) \mbox{ for all }t\in [t_0,T].$$
Setting
$$z(t)=\left\{\begin{array}{ll}
    \tilde{z}(t) &  \mbox{for }t\in [t_0,T],\\
    h(t) & \mbox{for }t\in [0,t_0],
\end{array}\right.$$
we have that $z\in C([0,T];\R^{d})$ and $D_y\Psi(t_0,y)z=h$. Moreover $z\in C([0,T];\R^d)$ depends continuously on $h\in C([0,T];\R^d)$.
Hence, by the implicit function theorem, the function $t_0\mapsto\Phi(t_0,y_0,v)$ is Fréchet differentiable from $\R$ to $C([0,T];\R^{d})$ at each $t_0\in (0,T)$ and its differential is given by \eqref{PhiPartial:t0}. For the Lipschitz continuity we observe that by using the Gronwall inequality in \eqref{PhiPartial:t0} we obtain that $z_1(t)$ is uniformly bounded in $t\in [0,T]$ and $t_0\in (0,T)$. This implies that $t_0\mapsto\Phi(t_0,y_0,v)$ is Lipschitz continuous on $[0,T]$ for each $(y_0,v)\in\Omega\times\mathcal{D}$.
\end{proof}
For $v\in \mathcal{D}$, by continuity the mapping $\Phi$ can be extended to the closure of $\mathcal{G}$ as asserted in the following proposition. We leave the verification to the reader.

\begin{prop}
Let $v\in\mathcal{D}$, the function $\Phi(\cdot,\cdot,v)$ is well defined in $\overline{\mathcal{G}(v)}$. Moreover, \eqref{PhiPartial:y0}, \eqref{PhiPartial:t0} and \eqref{PhiPartial:v} hold for all $(t_0,y_0)\in\overline{\mathcal{G}(v)}$ and $\phi\in C([0,T],C^{1}(\overline{\Omega}))$.
\label{prop:PhiExt}
\end{prop}
Clearly the differentiability of $\Phi$ implies the differentiability of $\mathcal{V}$. On the top of that, in the following proposition we provide a characterization the derivatives of $\mathcal{V}$, which is of vital importance for deriving and understanding the optimality condition of \eqref{LearningProblem}.
\begin{prop}
\label{lemma:Vauxregularity}
The function $\mathcal{V}$ is continuously differentiable $(0,T)\times\omega\times \mathcal{D}$. Further, letting $v\in \mathcal{D}$, $(t_0,y_0)\in (0,T)\times\omega$, and $y=\Phi(t_0,y_0,v)$, and defining the adjoint state $p\in C^{1}([t_0,T];\R^{d})$ as the solution of
\beq \left\{\begin{array}{l}
      \dis-p'(t)-D_y F_v^{\top}(t,y(t))  p(t)=-\nabla_{y}\ell(t,y(t))-\beta D_{y}U_{v}(t,y(t))^{\top}U_v(t,y(t)),\ \forall t\in[t_0,T],\\
      \ecart\dis p(T)=-\nabla g(y(T)),
\end{array}\right. \label{lemma:Vauxregularity:adjoint}\eeq
we have
\beq  \nabla_{y_0} \mathcal{V}(t_0,y_0,v)=-p(t_0),\label{lemma:Vauxregularity:ygrad} \eeq
\beq  \frac{\partial \mathcal{V}}{\partial t_0} (t_0,y_0,v)=-\ell(t_0,y_0)-\frac{\beta}{2}|U_v(t_0,y_0)|^2+p(t_0)^{\top}\cdot \left(f (t_0,y_0) +BU_v(t_0,y_0)\right),\label{lemma:Vauxregularity:tgrad} \eeq
and for $\phi\in  C([0,T];C^{1}(\overline{\Omega}))$
\beq D_v\mathcal{V}(t_0,y_0,v)\cdot\phi=\frac{1}{\beta}\int_{t_0}^{T}\nabla_y \phi(t,y(t)){^\top} BB^{\top}(\nabla_y v(t,y(t))+p(t))dt.\label{lemma:Vauxregularity:vgrad} \eeq

\end{prop}
\begin{proof}
Since $\Phi$ is continuously differentiable, we have that $\mathcal{V}$ is continuously differentiable. Moreover, for $t_0\in (0,T)$, $y_0\in \omega$ and $v\in\mathcal{D}$, we have that

\beq
\frac{d}{dt_0}\mathcal{V}(t_0,y_0,v)= \left(\begin{array}{l}\dis-\ell(t_0,y_0)-\frac{\beta}{2}\left|U_v(t_0,y_0)\right|^{2}+ z_1^{\top}(T)\nabla g(y(T))\\
\ecart\dis+\int_{t_0}^{T}z_1^{\top}(t)\left(\nabla_{y} \ell (t,y(t))+\beta D_{y}U_{v}(t,y(t))^{\top} U_{v}(t,y(t))\right) dt,
\end{array} \right) \label{lemma:Vauxregularity:proof:z1}\eeq
\beq \nabla_{y_0} \mathcal{V}(t_0,y_0,v)=\int_{t_0}^{T}z_2^{\top}(t)\left(\nabla_{y} \ell (t,y(t))+\beta D_{y}U_{v}(t,y(t))^{\top} U_{v}(t,y(t))\right) dt+z_2^{\top}(T)\nabla g(y(T)) \label{lemma:Vauxregularity:proof:z2} \eeq

\beq D_{v}\mathcal{V}(t_0,y_0,v)\cdot\phi=\left(\begin{array}{l}
     \dis \int_{t_0}^{T}z_3^{\top}(t)\left(\nabla_{y} \ell (t,y(t))+\beta D_{y}U_{v}(t,y(t))^{\top} U_{v}(t,y(t))\right) dt\\
     \ecart\dis+\beta\int_{t_0}^{T} U_{\phi}(t,y(t))^{\top}U_{v}(t,y(t))dt+ z_3^{\top}(T)\nabla g(y(T))
\end{array}\right) \label{lemma:Vauxregularity:proof:z3}
\eeq
where $y=\Phi(t_0,y_0,v)$, $z_1=\frac{d}{dt_0} \Phi(t_0,y_0,v)$, $z_2=D_{y_0} \Phi(t_0,y_0,v)$ and $z_3=D_{v} \Phi(t_0,y_0,v)\cdot\phi$. We know that $z_1$, $z_2$ and $z_3$ satisfy \eqref{PhiPartial:t0},\eqref{PhiPartial:y0} and \eqref{PhiPartial:v}. Then, for $t\in (t_0,T)$ and $i\in\{1,2\}$ we have
\beq \frac{d}{dt}z_{i}(t)=D_{y}F(t,y(t))z_{i}.\label{lemma:Vauxregularity:proof:zieq}\eeq
and
\beq z_1(t_0)=-F_{v}(t_0,y_0) \mbox{ and } \quad z_2(t_0)=I. \label{lemma:Vauxregularity:proof:zinitcond}\eeq
For $z_3$ we know
\beq \frac{d}{dt}z_{3}(t)=D_{y}F(t,y(t))z_{i}+B U_\phi(t,y(t)),\quad z_{3}(t_0)=0.\label{lemma:Vauxregularity:proof:z3eq}\eeq
Multiplying  \eqref{lemma:Vauxregularity:adjoint} by $z_i$ for $i\in \{1,2\}$, using \eqref{lemma:Vauxregularity:proof:zieq}, and rearranging the terms, we obtain
\beq \int_{t_0}^{T}z_i^{\top}(t)\left(\nabla_{y} \ell (t,y(t))+\beta D_{y}U_{v}(t,y(t))^{\top} U_{v}(t,y(t))\right) dt+ z_i^{\top}(T)\nabla g(y(T))=z_i(t_0)p(t_0).\eeq
Using this and \eqref{lemma:Vauxregularity:proof:zinitcond} in \eqref{lemma:Vauxregularity:proof:z1} and \eqref{lemma:Vauxregularity:proof:z2}, we obtain \eqref{lemma:Vauxregularity:tgrad} and \eqref{lemma:Vauxregularity:ygrad}. Analogously, for $z_3$ we have
\beq \begin{split}
\int_{t_0}^{T}z_3^{\top}(t)\left(\nabla_{y} \ell (t,y(t))+\beta D_{y}U_v(t,y(y))^{\top} U_{v}(t,y(t))\right) dt+ z_3^{\top}(T)\nabla g(y(T))\\=z_3(t_0)p(t_0)-\int_{t_0}^{T}U_{\phi}(t,y(t))^{\top}B^{\top}p(t) dt.\end{split}\eeq
Combining this with the definition of $U_v$ and $U_\phi$, and \eqref{lemma:Vauxregularity:proof:z3}, we obtain \eqref{lemma:Vauxregularity:vgrad}.
\end{proof}
Further, for a fixed $v\in \mathcal{D}$, by means of \Cref{prop:PhiExt} it easy to see that $\mathcal{V}(\cdot,\cdot,v)$ can be extended to $\mathcal{G}(v)$ as is stated in the following proposition.
\begin{prop}
\label{prop:derexten}
Let $v\in\mathcal{D}$, the function $\mathcal{V}(\cdot,\cdot,v)$ is well defined in $\overline{\mathcal{G}(v)}$. Moreover, \eqref{lemma:Vauxregularity:ygrad}, \eqref{lemma:Vauxregularity:tgrad} and \eqref{lemma:Vauxregularity:vgrad} hold for all $(t_0,y_0)\in \overline{\mathcal{G}(v)}$ and $\phi\in C([0,T];C^{1}(\overline{\Omega}))$.
\end{prop}
As a consequence of the previous result and the chain rule we obtain the following proposition.
\begin{prop}
\label{prop:Jregularity}
Let $v\in \mathcal{D}$. Then $\mathcal{J}$ is continuously differentiable in $v$ and the differential of $\mathcal{J}$ in $v$ is given by
\beq D\mathcal{J}(v)\cdot\phi=\frac{1}{T|\omega|}\int_{\omega}\int_{0}^{T}D_v\mathcal{V}(t_0,y_0,v)\cdot\phi\, dt_0dy_0\mbox{ for all }\phi\in C([0,T];C^{1}(\overline{\Omega})). \label{prop:Jregularity:eq}\eeq
\end{prop}
Consequently the necessary optimality condition for $v^{*}\in C([0,T];C^{2}(\overline{\Omega}))$ is
\beq \int_{\omega}\int_{0}^{T}D_v\mathcal{V}(t_0,y_0,v^{*})\cdot\phi \,dt_0dy_0=0 \mbox{ for all }\phi\in C([0,T];C^{1}(\overline{\Omega})).\label{necessarycond}\eeq
In the subsequent proposition we prove that the necessary condition implies that $\mathcal{V}(\cdot,\cdot,v^{*})$ satisfies HJB equation in $\mathcal{G}(v^*)$. A remarkable consequence of this is that the necessary conditions of optimality becomes sufficient if the set $\mathcal{G}(v^*)$ is large enough.
\begin{prop}
\label{prop:sufficientcond} Let $v^{*}\in \mathcal{D}$ satisfy   \eqref{necessarycond}. Then $\mathcal{V}(\cdot,\cdot,v^{*})$ satisfies equation \eqref{HJBeq} in $\mathcal{G}(v^*)$ and $\mathcal{V}(T,y_0,v^{*})=g(y_0)$ for all $y_0$ such that $(T,y_0)\in \overline{\mathcal{G}(v^*)}$. Further, if for $(t_0,y_0)\in \mathcal{G}(v^{*})$ there exists an optimal trajectory of \eqref{ControlProblem}, denoted by $y^{*}\in C^{1}([0,T];\R^{d})$, which satisfies $(t,y^*(t))\in \mathcal{G}(v^{*})$ for all $t\in[t_0,T]$, then $U_{v^{*}}$ is an optimal feedback for \eqref{ControlProblem}.
\end{prop}

\begin{proof}
By \Cref{prop:derexten}  we can use \eqref{lemma:Vauxregularity:ygrad}
and \eqref{lemma:Vauxregularity:tgrad} to obtain that for all
\beq \nabla_{y_0}\mathcal{V}(t,y(t;t_0,y_0),v^{*})=-p(t;t_0,y_0)\label{prop:sufficientcond:proof:ygrad}
\eeq
and
\beq \begin{array}{l} \dis\frac{d \mathcal{V}}{dt_0}(t,y(t;t_0,y_0),v^{*})+\ell(t,y(t;t_0,y_0))+\frac{\beta}{2}\left|U_{v^*}(t,y(t;t_0,y_0))\right|^2\\
\ecart\dis=-p(t;t_0,y_0)^{\top} (f(t,y(t;t_0,y_0))+BU_{v^{*}}(t,y(t;t_0,y_0))).\end{array}\label{prop:sufficientcond:proof:tgrad}
\eeq
for all $(t_0,y_0)\in (0,T)\times\omega$ and for all $t\in [t_0,T]$, where $y(\cdot;t_0,y_0)$ and $p(\cdot;t_0,y_0)$ are the solutions of \eqref{closedloop2} and \eqref{lemma:Vauxregularity:adjoint} with $v=v^*$. Setting $\phi=v^*-\mathcal{V}$, plugging it into \eqref{lemma:Vauxregularity:vgrad}, and using \eqref{prop:sufficientcond:proof:ygrad} we have
$$ \int_{\omega}\int_{0}^{T}\int_{t_0}^{T}\left|B^{\top}\left(\nabla_{y}v^{*}(t,y(t;t_0,y_0))+p(t;t_0,y_0)\right)\right|^{2}dt_0dtdy_0=0\label{prop:sufficientcond:proof:vpeq}$$
which implies
$$ B^{\top}\nabla_y v^{*}(t,y(t;t_0,y_0))=-B^{\top}p(t;t_0,y_0)\mbox{ for all }t\in [t_0,T].$$
Hence, by the definition of $U_v^{*}$ and \eqref{prop:sufficientcond:proof:ygrad} we have
$$U_v^{*}(t,y(t;t_0,y_0))=-\frac{1}{\beta}B^{\top}\nabla_y v^{*}(t,y(t;t_0,y_0))=-\frac{1}{\beta}B^{\top}\nabla_{y_0} \mathcal{V}(t,y(t;t_0,y_0),v^{*}) $$
Combining this with \eqref{prop:sufficientcond:proof:tgrad} and rearranging the terms, we get
$$\begin{array}{l} \dis\frac{d \mathcal{V}}{dt_0}(t,y(t;t_0,y_0),v^{*})+\ell(t,y(t;t_0,y_0))\\
\ecart\dis=\frac{1}{2\beta}\left|B^{\top}\nabla_{y_0} \mathcal{V}(t,y(t;t_0,y_0),v^{*})\right|^2-\nabla_{y_0} \mathcal{V}(t,y(t;t_0,y_0),v^{*})^{\top} f(t,y(t;t_0,y_0)),\end{array}$$
which is equivalent to
$$\begin{array}{l} \dis\frac{d \mathcal{V}}{dt_0}(t,y(t;t_0,y_0),v^{*})+\ell(t,y(t;t_0,y_0))\\
\ecart \dis =-\min_{U\in\R^{M}}\left\{\nabla_{y_0}\mathcal{V}(t,y(t;t_0,y_0),v^{*})^{\top}(f(t,y(t;t_0,y_0))+BU)+\frac{\beta}{2}|U|^2\right\}.\end{array}$$
Since this holds for all $(t_0,y_0)\in (0,T)\times\omega$ and $t\in[t_0,T]$, it follows that $\mathcal{V}(\cdot,\cdot,v^{*})$ satisfies HJB equation in $\mathcal{G}(v^{*})$. Finally, assume that $(t_0,y_0)\in \mathcal{G}(v^{*})$ and that there exists an optimal trajectory of \eqref{ControlProblem}, denoted by $y^{*}$, that satisfies $(t,y^{*}(t))\in \mathcal{G}(v^{*})$. Then, by the verification theorem we have that
$$ \mathcal{V}(t_0,y_0,v^{*})\leq V(t_0,y_0)$$
and therefore $U_{v^{*}}$ is an optimal feedback for \eqref{ControlProblem}.
\end{proof}
\begin{rem}
This theorem ensures that if a function $v^{*}\in \mathcal{D}$ satisfies the optimality condition \eqref{necessarycond} and the set $\overline{\mathcal{G}(v^{*})}$ is large enough to contain at least one optimal trajectory of \eqref{ControlProblem} for all $(t_0,y_0)\in(0,T)\times\Omega$, then the feedback $U_{v^{*}}$  is an optimal solution of the learning problem \eqref{LearningProblem}.
\end{rem}

\section{Finite dimensional Learning problem}
\label{FiniteDimensionalLearningProblem}
Herein we introduce a discrete approximation of \eqref{LearningProblem}. Let $X$ be a finite dimensional real linear space and a continuous injection $\mathcal{I}:X\mapsto C^{1}([0,T];C^{2}(\overline{\Omega}))$. For simplicity we assume that $X=\R^{N}$. We define
\beq \min_{\theta\in X,\mathcal{I} (\theta)\in \overline{\mathcal{D}}} \mathcal{J}(\mathcal{I} (\theta))+\gamma\left(\frac{1-r}{2}|\theta|^{2}_{2}
     +r|\theta|_{1}\right) \label{AppLearningProblem}\eeq
where $\alpha>0$ and $r\in [0,1]$ are penalty coefficients. The penalty term $\gamma\left(\frac{1-r}{2}|\theta|^{2}_{2}+r|\theta|_{1}\right)$  ensures the coercivity of the objective function. Moreover the $\ell_{1}$ term promotes the sparsity of the solution of \eqref{AppLearningProblem}. However, unless we assume some further hypotheses on the structure of $\ell$, $f$, $B$ and/or the value function $V$, we do not yet know if there exist a solution of \eqref{AppLearningProblem}. \par
We now discuss the existence of solutions for \eqref{AppLearningProblem} and under which conditions this problem is converging to \eqref{LearningProblem}. We say that \eqref{AppLearningProblem} is feasible if there exists $\theta\in\R^{m}$ such that $\mathcal{I}(\theta)\in  \overline{\mathcal{D}}$. It is clear that the objective function of \eqref{AppLearningProblem} is continuous and coercive. Then, since the set $\{\theta\in X:\ \mathcal{I}(\theta)\in \overline{\mathcal{D}}\}$ is closed we can use the direct method of the calculus of variations to deduce that problem \eqref{AppLearningProblem} has at least one solution. Now, we analyze the convergence of \eqref{AppLearningProblem} to \eqref{LearningProblem}. For this purpose, let us consider the following hypothesis for a sequence of finite subsets of $C([0,T];C^{1}(\overline{\Omega}))$
\begin{hypo}
Let $\{(X_{k},\mathcal{I}_k)\}_{k=1}^{\infty}$ such that, for all $k\geq 0$, $X_{k}$ is a finite dimensional linear space, $\mathcal{I}_k:X_k\mapsto C([0,T];C^{2}(\overline{\Omega}))$ is continuous and injective,  and that there exists a sequence $\vartheta_{k}\in X_k$ satisfying
\beq \lim_{k\to \infty}\norm{\mathcal{I}_{k}(\vartheta_k) -V}_{C([0,T];C^{1}(\overline{\Omega}))}=0. \eeq
\label{Hyp:Seq:Dens}
\end{hypo}
Under \cref{Hyp:Seq:Dens} we obtain the following convergence result for \eqref{AppLearningProblem}.
\begin{theo}
Assume that $\{X_{k}\}_{k=1}^{\infty}$ is a sequence satisfying \Cref{Hyp:Seq:Dens} and that $V\in \mathcal{D}$. Then, there exists $k_{0}\in \N$ such that for all $k\geq k_{0}$, $\gamma>0$ and $r\in[0,1]$ problem \eqref{AppLearningProblem} has at least one solution with $X=X_{k}$. Further, for all $r\in [0,1]$, there exists a sequence of penalty coefficients $\gamma_{k}>0$ (possible converging to 0) such that the following hold
\begin{itemize}
    \item there exits a constant $C>0$ independent of $k$, which depends on $\norm{f}_{C([0,T]\times\Omega)}$, $\norm{V}_{C([0,T];C^{1}(\overline{\Omega}))}$, $|B|$, $\beta$, $\Omega$ and $T$ such that 
    \beq \int_{\omega}\int_{0}^{T}\left|\mathcal{V}(t_0,y_0,\mathcal{I}_k(\theta_{k}))-\mathcal{V}(t_0,y_0,V)\right|dt_0 dy_0\leq C\norm{\mathcal{I}_k(\vartheta_{k})-V}_{C([0,T];C^1(\overline{\Omega}))}, \label{Theo:convergence:L1conv}\eeq
    
    \item for almost every $(t_0,y_0)\in (0,T)\times\omega$, setting
    \beq y_{k}=\Phi(t_0,y_0,\mathcal{I}_k(\theta_{k})) \mbox{ and } u_{k}(t)=U_{\mathcal{I}_k(\theta_{k})}(t,y_{k}(t)), \eeq
    there exists a weakly convergent sub-sequence of $(u_{k})_{k\geq k_0}$ in  $L^{2}((0,T);\R^{M})$. Moreover, every such sequence converges strongly in $L^{2}((0,T);\R^{M})$ to an optimal control of the open loop problem \eqref{ControlProblem},
\end{itemize}
where for each $k\geq k_0$, $\theta_{k}\in X_k$ is an optimal solution of \eqref{AppLearningProblem} with $X=X_{k}$ and $\gamma=\gamma_k$, and $\vartheta_k\in X_k$ is the sequence that appears in \Cref{Hyp:Seq:Dens}.
\label{Theo:convergence}
\end{theo}
\begin{proof}
Let $\vartheta_{k}\in X_k$ be a sequence such that
\beq \lim_{k\to \infty} \norm{\mathcal{I}_k(\vartheta_{k})-V}_{C([0,T];C^{1}(\overline{\Omega}))} =0.\label{Theo:convergence:proof:conV}\eeq
Since $V\in \mathcal{D}$, we can combine \Cref{lemma:OpenD} and \eqref{Theo:convergence:proof:conV} to deduce that there exists $k_0\in \N$ such that $\mathcal{I}_k(\vartheta_{k})\in \mathcal{D}$ for all $k\geq k_0$. Then, for all $k\geq k_0$ problem \eqref{AppLearningProblem} has at least one solution  with $X=X_{k}$. In the following we assume that $k\geq k_0$ and $\norm{\mathcal{I}_k(\vartheta_{k})-V}_{C([0,T];C^{1}(\overline{\Omega}))}\leq 1$.\par
Let $(t_0,y_0)\in (0,T)\times\omega$ and set $y=\Phi(t_0,y_0,\mathcal{I}_k(\vartheta_{k}))$ and $u(t)=U_{\mathcal{I}_k(\vartheta_{k})}(t,y(t))$. We have
\beq
\begin{array}{l}
\dis \frac{\beta}{2}|u(t)|^{2}+(f(t,y(t))+Bu(t))^{\top}\nabla_y \mathcal{I}_k(\vartheta_{k})(t,y(t)) \\
\ecart\dis \leq \frac{\beta}{2}|U|^{2}+(f(t,y(t))+BU)^{\top}\nabla_y \mathcal{I}_k(\vartheta_{k})(t,y(t))
\end{array}
 \label{Theo:convergence:proof:mincond} \eeq
for all $t\in (t_0,T)$ and $U\in \R^{d}$. By replacing $U$ by $U_V$ in \eqref{Theo:convergence:proof:mincond} and using that $V$ satisfies \eqref{HJBeq}, we have
$$\begin{array}{l}
    \dis \frac{\beta}{2}|u(t)|^{2}+(f(t,y(t))+Bu(t))^{\top}\nabla_y \mathcal{I}_k(\vartheta_{k})(t,y(t))\\
     \ecart \dis \leq \Big(-\frac{dV}{dt}(t,y(t))-\ell(t,y(t))+\\
     \ecart \dis
     (f(t,y(t))+BU_V(t,y(t)))^{\top}\left(\nabla_y \mathcal{I}_k(\vartheta_{k})(t,y(t))-\nabla_y V(t,y(t))\right)\Big).
\end{array}$$
Setting $K=\dis\sup_{z\in\Omega,t\in [0,T]}|F_{V}(t,z)|$ and using it in the previous inequality we obtain
\beq \begin{array}{l}\dis\frac{\beta}{2}|u(t)|^{2}+(f(t,y(t))+Bu(t))^{\top}\nabla_y \mathcal{I}_k(\vartheta_{k})(t,y(t))\\
\ecart\dis\leq -\frac{dV}{dt}(t,y(t))-\ell(t,y(t))+K \norm{\mathcal{I}_k(\vartheta_{k})-V}_{C([0,T];C^{1}(\overline{\Omega}))}.\end{array}
\label{Theo:convergence:proof:errin1}\eeq
Note that
$$ \begin{array}{l}
     \left|(f(t,y(t))+Bu(t))^{\top}\left(\nabla_y V(t,y(t))-\nabla_y\mathcal{I}_k(\vartheta_k)(t,y(t))\right)\right|\\
     \ecart\dis \leq K\norm{\mathcal{I}_k(\vartheta_{k})-V}_{C([0,T];C^{1}(\overline{\Omega}))}+
     \frac{|B|^{2}}{\beta}\norm{\mathcal{I}_k(\vartheta_{k})-V}_{C([0,T];C^{1}(\overline{\Omega}))}^{2},
\end{array} $$
which is obtained by expressing $u$ as $u=U_V+u-U_V$.
Subtracting $(f(t,y(t))+Bu(t))^{\top}\nabla_y V(t,y(t))$ from both sides of \eqref{Theo:convergence:proof:errin1}, using the above inequality and rearranging the term we find
$$\begin{array}{l}\dis\frac{\beta}{2}|u(t)|^{2}+\ell(t,y(t))
\dis\leq -\frac{dV}{dt}(t,y(t))-(f(t,y(t))\\[1.5ex]
\hspace{3.5cm}+Bu(t))^{\top}\nabla_y V(t,y(t))+\left(2K+\frac{|B|^{2}}{\beta}\right) \norm{\mathcal{I}_k(\vartheta_{k})-V}_{C([0,T];C^{1}(\overline{\Omega}))}.\end{array}$$
Integrating from $t_0$ to $T$, we get
$$\begin{array}{l} \int_{t_0}^{T}l(t,y(t))+\frac{\beta}{2}|u(t)|^{2}dt\\[1.5ex]
\qquad \leq  V(t_0,y_0)-V(T,y(T))+\left(2K+\frac{|B|^{2}}{\beta}\right)(T-t_0) \norm{\mathcal{I}_k(\vartheta_k)-V}_{C([0,T];C^{1}(\overline{\Omega}))}.\end{array}$$
Since $V(T,y(T))=g(T,y(T))$ and the previous inequality holds for all $(t_0,y_0)\in (0,T)\times\omega$, we can integrate with respect to $t_0$ in $[0,T]$ and $y_0$ in $\omega$ and divide by $T|\omega|$ to obtain
\beq \mathcal{J}(\mathcal{I}_k(\vartheta_k))\leq \mathcal{J}(V) +T \left(2K+\frac{|B|^{2}}{\beta}\right)\norm{\mathcal{I}_k(\vartheta_k)-V}_{C([0,T];C^{1}(\overline{\Omega}))}. \label{Theo:convergence:proof:auxConv}\eeq
\par
For a fixed $r\in [0,1]$, let $\theta_{k,\gamma}\in \R^{m_k}$ be the solution of \eqref{AppLearningProblem} for $\gamma>0$ with $X=X_{k}$. Then, by the optimality of $\theta_{k,\gamma}$ and \eqref{Theo:convergence:proof:auxConv}, we have
\beq \begin{array}{l}
     \dis \mathcal{J}(\mathcal{I}_k(\theta_{k,\gamma}))+\gamma\left(\frac{1-r}{2}|\theta_{k,\gamma}|_2^{2}+|\theta_{k,\gamma}|_1\right)\leq \mathcal{J}(\mathcal{I}_k(\vartheta_{k}))+\gamma\left(\frac{1-r}{2}|\vartheta_{k}|_2^{2}+l|\vartheta_{k}|_1\right)\\
     \ecart \dis \leq \mathcal{J}(V) + T\left(2K+\frac{|B|^{2}}{\beta}\right)\norm{\mathcal{I}_k(\vartheta_{k})-V}_{C([0,T];C^{1}(\overline{\Omega}))}+\gamma\left(\frac{1-r}{2}|\vartheta_{k}|_2^{2}+l|\vartheta_{k}|_1\right).
\end{array}\eeq
Now we make the announced choice of $\gamma_k$ as
$$\gamma_k=\left(2K+\frac{|B|^{2}}{\beta}\right)T\frac{\norm{\mathcal{I}_k(\vartheta_{k})-V}_{C([0,T];C^{1}
(\overline{\Omega}))}}{\frac{1-r}{2}|\vartheta_{k}|_2^{2}+l|\vartheta_{k}|_1}$$
from which we obtain that
\beq  \mathcal{J}(\mathcal{I}_k(\theta_{k,\gamma})\leq \mathcal{J}(V) +2 T \left(2K+\frac{|B|^{2}}{\beta}\right) \norm{\mathcal{I}_k(\vartheta_{k})-V}_{C([0,T];C^{1}(\overline{\Omega}))}.
\eeq
Taking $k\to\infty$ we obtain that $\mathcal{J}(\mathcal{I}_k(\theta_{k,\gamma})$ converges to $\mathcal{J}(V)$. Further, by the definition of $\mathcal{V}$ and $\mathcal{J}$ we have
\beq \begin{array}{l}
\dis\int_{\omega}\int_{0}^{T}\left|\mathcal{V}(t_0,y_0,\mathcal{I}_k(\theta_{k,\gamma}))-V(t_0,y_0)\right|dt_0 dy_0\\
\ecart\dis=\int_{\omega}\int_{0}^{T}\left(\mathcal{V}(t_0,y_0,\mathcal{I}_k(\theta_{k,\gamma}))-V(t_0,y_0)\right)dt_0 dy_0\\
\ecart\dis\leq 2 T|\omega| \left(K+\frac{|B|^{2}}{\beta}\right) \norm{\mathcal{I}_k(\vartheta_{k})-V}_{C([0,T];C^{1}(\overline{\Omega}))}.
\end{array}   \eeq
This implies \eqref{Theo:convergence:L1conv} and consequently $\mathcal{V}(\cdot,\cdot,\mathcal{I}_k(\theta_{k,\gamma}))$ converges to $V$ in $L^{1}((0,T)\times\omega)$ as $k$ tends to infinity. Further, we can extract a sub-sequence of $\theta_{k,\gamma}$ still denoted by $\theta_{k,\gamma}$, such that $\mathcal{V}(\cdot,\cdot\mathcal{I}_k(\theta_{k,\gamma}))$ converges to $V$ almost everywhere.\par
Now, let  $(t_0,y_0)\in (0,T)\times\omega$ be such that $\mathcal{V}(t_0,y_0,\mathcal{I}_k(\theta_{k,\gamma}))$ converges to $V(t_0,y_0)$. Setting $y_k=\Phi(t_0,y_0,\mathcal{I}_k(\theta_{k,\gamma}))$ and $u_k=U_{\mathcal{I}_k(\theta_{k,\gamma})}$, we have
\beq y_k(t)\in \Omega\mbox{ for all }t\in [t_0,T]\mbox{ and }\norm{u_k}_{L^{2}((0,T);\R^{M})}^{2}\leq \frac{2}{\beta}\mathcal{V}(t_0,y_0,\mathcal{I}_k(\theta_{k,\gamma})) \eeq
for all $k\geq k_0$. Since $\mathcal{V}(t_0,y_0,\mathcal{I}_k(\theta_{k,\gamma})) $ is converging to $V(t_0,y_0)$, we have that $u_k$ is bounded in $L^{2}((t_0,T);\R^{M})$. Additionally, due to the fact that $\Omega$ is bounded, we know that $y_k$ is bounded in $C([t_0,T];\overline{\Omega})$. Combining the boundness of $u_k$  in $L^{2}((t_0,T);\R^{M})$ and boundness of $y_k$  in $C([t_0,T];\overline{\Omega})$ with the fact that $f$ is continuous in $\overline{\Omega}$ we get that $y$ is bounded in $H^{1}((t_0,T);\R^{d})$. Hence, passing to a sub-sequence, we obtain that there exist $\hat{u}\in L^{2}((t_0,T),\R^{M})$ and $\hat{y}\in H^{1}((t_0,T);\R^{d})$ such that
$$ u_k\weak\hat{u}\mbox{ in }L^{2}((t_0,T),\R^{M})\mbox{ and }y_k\weak\hat{y}\mbox{ in }H^{1}((t_0,T),\R^{d}), $$
when $k\to\infty$. Further, the compact embedding of $H^{1}((t_0,T),\R^{d})$ in $C([t_0,T];\R^{d})$ implies that, passing to sub-sequence,
$$ y_k\to\hat{y}\mbox{ in }C([t_0,T];\overline{\Omega}). $$
This convergences and the continuity of $\ell$ and $g$, and the lower semicontinuity of the norm of $L^{2}((t_0,T),\R^{M})$ imply that $\hat{u}$ is a solution of \eqref{ControlProblem} and $\hat{y}$ is the respective optimal trajectory. The strong convergence of $u_k$ is a consequence of its weak convergence and
the following norm convergence
\beq\begin{array}{l}
\dis\norm{u_k}^{2}_{L^{2}((t_0,T);\R^{M})}=\frac{2}{\beta}\left(\mathcal{V}(t_0,y_0,\mathcal{I}_k(\theta_{k,\gamma}))-\int_{t_0}^{T}\ell(t,y_k(t))dt\right)\\
\dis \to \frac{2}{\beta}\left(V(t_0,y_0)-\int_{t_0}^{T}\ell(t,\hat{y}(t))dt\right)=\norm{\hat{u}}^{2}_{L^{2}((t_0,T);\R^{M})},\end{array}  \eeq
for $k\to\infty$. Since for each $k\geq k_0$ $y_k$ satisfies \eqref{OpenLoop}, the strong convergence of $u_k$ implies that $y_k$ converges strongly to $\hat{y}$ in $H^{1}((t_0,T);\R^{d})$.
\end{proof}
\section{Optimization algorithm}
\label{OptimizationAlg}
In this section we describe the optimization algorithm that we use to solve \eqref{AppLearningProblem}. To this end, we consider $X=\R^{N}$ and a differentiable injection $\mathcal{I}:X\mapsto C([0,T];C^{2}(\overline{\Omega}))$. With the last assumption, the function $\mathcal{J}\circ \mathcal{I}$ is differentiable. Nevertheless, the objective function of problem \eqref{AppLearningProblem} includes a $\ell_1$ penalty term to promote the sparcity, which is non differentiable. Taking this into account, we decided to use a proximal gradient method. Additionally, we choose the Barzilai Borwein step size along with a non-monotone backtracking line search, which proved to be efficient for high dimensional problems (see \cite{AzKK}, \cite{Barzilai}, and \cite{Raydan} for a convergence analysis in the smooth case).
\par
We now describe the algorithm that we use to solve \eqref{AppLearningProblem}. We set $\tilde{\mathcal{J}}=\mathcal{J}\circ \mathcal{I}$ and we denote the $k-$th element of the sequence produced by the algorithm by $\theta^{k}\in X$, the step size by $s^k$, and we set at each iteration
\beq d^{k}:=\nabla \tilde{\mathcal{J}}(\theta^{k})+\gamma(1-r)\theta^{k}.
\label{gradient}\eeq
We use the proximal point update rule as is described  in section 10.2 in \cite{Beck}, namely we choose $\theta^{k+1}$ such that
\begin{equation}
    \theta^{k+1}=\argmin_{\vartheta\in X}
    \left\{d^{k}\cdot\left(\vartheta-\theta^k\right) +\frac{1}{2s^k}|\theta^k-\vartheta|_{2}^{2}+\gamma r|\vartheta|_{1}\right\}.
    \label{ProximalPointUpdate}
\end{equation}
To choose the step size $s^k$ we use the non-monotone backtracking line search described in \cite{WrightNowakFig}, starting from an initial guess $s_{0}^{k}$.
 That is, for $\kappa\in (0,1)$, $n$ a non negative integer number and $\beta\in (0,1)$, we take $s^{k}=s_0^{k}\beta^{i}$ such that $i$ is the smallest natural number which satisfies
\beq \tilde{\mathcal{J}}(\theta^{+}) \leq \max_{i\in\{\max(0,k-n),\ldots,n\}}\tilde{\mathcal{J}}(\theta^{k-i}) - \frac{\kappa}{s_0^{k}\beta^{i}}|\theta^{k}-\theta^{+}|^2, \label{Backtracking}\eeq
for $\theta^+>0$ and set $\theta^{k+1}=\theta^+$.
The Barzilai-Borwein step size is used as initial guess, namely we take $s_0^k$ as
\beq  s_{0}^{k}=\left\{\begin{array}{ll}
     \dis \big[(\theta_{k}-\theta_{k-1})\cdot(d_{k}-d_{k-1})\big]/|d_{k}-d_{k-1}|^{2}   &\mbox{ if } k\mbox{ is }odd,\\
      \ecart\dis |\theta_{k}-\theta_{k-1}|^{2}/\big[(\theta_{k}-\theta_{k-1})\cdot(d_{k}-d_{k-1})\big] & \mbox{ if }k\mbox{ is }even.
\end{array}\right. \label{BB}\eeq
We use non-monotone stopping criterion, namely, defining for $k>0$
$$J_k:=\tilde{\mathcal{J}}(\theta^{k})+\gamma\left(\frac{1-r}{2}|\theta^{k}|^2_2+r|\theta^{k-i}|_1\right) \mbox{ and }\quad J^{max}_{k}:=\min_{i\in\{\max(0,k-n),\ldots,n\}}J_{k-i},$$
we iterate until $\left(J^{max}_{k-1}-J_k\right)\leq J^{max}_{k-1}\cdot tol $.\par
We summarize the algorithm as follows:
\begin{algorithm}[H]
\caption{Learning algorithm.}
\label{Alg1}
\begin{algorithmic}[1]
\Require An initial guess $\theta^{0}\in\R^{M}$, $\kappa>0$, $\beta\in (0,1)$, $s_0\in (0,\infty)$.
\Ensure An approximated stationary point $\theta^{*}$ of \eqref{AppLearningProblem}.
\State $k=1 $
\State For $\theta^0$, set $d^0$ \eqref{gradient}
and $J_0$.
\State Use \eqref{BB} to obtain $s_{0}$.
\State Use \eqref{ProximalPointUpdate} to get $\theta^{1}$.
\State Obtain $d^{1}$ and set $J_1$
\While{ $J^{max}_{k-1}-J_{k}>tol\cdot J^{max}_{k-1} $ }
\State Obtain $s_{0}^{k}$ by using \eqref{BB} and  choose $s^{k}$ using \eqref{Backtracking}.
\State Use \eqref{ProximalPointUpdate}  to get $\theta^{k+1}$.
\State Obtain $d^{k+1}$ by using \eqref{gradient} and set $J_{k+1}=\tilde{\mathcal{J}}(\theta^{k+1})+\gamma\left(\frac{1-r}{2}|\theta^{k+1}|^2_2+r|\theta^{k+1}|_1\right).$
\State Set $k=k+1$.
\EndWhile
\Return $\theta^{\bar{k}}$, such that $\bar{k}\in \argmin_{k}J_k$.

\end{algorithmic}

\end{algorithm}

\section{Polynomial learning problem}
\label{PolynomialLearningProblem}
In order to introduce the polynomial ansatz we need some notation. Let  $n\in \N$ and $d\in \N$, where $\N=\{0,1,2,\ldots\}$. For a multi-index $\alpha=(\alpha_1,\ldots,\alpha_d)\in \N^{d}$ we define a monomial $\phi_{\alpha}$ by
\beq \phi_\alpha(y)=\prod_{j=1}^{d}y_{j}^{\alpha_j},\quad y \in \R^{d}.\label{Monomial}\eeq
We denote by $\Lambda_{n}$ the set of multi-indexes such the sum of all its elements is lower or equal than $n$ and  by $\B_{n}$ the set of all the monomials with total degree lower or equal to $n$, that is
\beq\Lambda_n=\left\{ \alpha\in \N^{d}:\ \sum_{j=1}^{d}\alpha_j\leq n\right\},\quad \mathcal{B}_{n}=\{\phi_{\alpha}:\ \alpha\in \Lambda_{n}\} \label{AlphaIndexes}\eeq

We further denote  the {\em hyperbolic cross} multi-index set by $\Gamma_{n}$ and we also introduce a subset $\mathcal{S}_{n}$ of $\B_n$ composed by the elements of  the subset $\B_n$ associated to the multi-indexes in $\Gamma_n$, i.e.
\beq
\Gamma_{n}=\Big\{\alpha=(\alpha_1,\ldots,\alpha_d)\in \N^{d}:\quad \prod_{j=1}^{d}(\alpha_j+1)\leq n+1\Big\}, \quad \mathcal{S}_{n}=\{\phi_{\alpha}:\ \alpha\in\Gamma_{n}  \}.
\label{HyperCrossMultiindex}
\eeq
\par

It is important to observe that the cardinality of $\Lambda_n$ is $\sum_{j=1}^{n}\binom{d+j+1}{ j}$, on the other hand the cardinality of $\Gamma_n$ is bounded by $\min\{2n^{3}4^{d},e^{2}n^{2+\log_2(n)}\}$ (see \cite{Adcock2017}). Hence,  for high $d$ the cardinality of the hyperbolic cross is smaller than the cardinality of $\Lambda_n$. For this reason, $\Lambda_n$ is more suitable for high dimensional problems.
\par
We point out that we only need to consider the elements in either $\Lambda_n$ or $\Gamma_n$ that satisfy $B^{\top}\nabla\phi_{\alpha}(y)\neq 0$ for all $y\in\R^{d}$. By Lemma 1 in \cite{KuVaWal}, we know that these elements are characterized by the condition $B^{\top}e_{i}\neq 0$ for all $i=1,\ldots,d$ such that $\alpha_i>0$. For a set $S$ of multi-indexes we define
\beq S(B)=\{\alpha\in S:\ B^{\top}e_{i}\neq 0 \mbox{ for all }i=1,\ldots,d \mbox{ such that } \alpha_i>0\}.\eeq
With this definition, we only consider either $\Gamma_n(B)$ or $\Lambda_n(B)$.
\par
We assume that $\Gamma_n(B)$ and $\Lambda_n(B)$ are ordered, namely $\Lambda_n(B)=\{\alpha^{tot}_i\}_{i=1}^{|\Lambda_n(B)|}$ and $\Gamma_n(B)=\{\alpha_i^{hc}\}_{i=1}^{|\Gamma_n(B)|}$. For $j\in \{0,1,\ldots,m\}$ and $t\in \R$ we define
$q_{j}(t)=t^{j}$. Then, for $m\in\N$ we define $\mathcal{I}^{tot}_{m,n}(\theta)$ and
$\mathcal{I}^{hc}_{m,n}(\theta)$ by
\beq \mathcal{I}^{tot}_{m,n}(\theta)(t,y)=\sum_{j=0}^{m-1}\sum_{i=1}^{|\Lambda_n(B)|}\theta_{j,i}q_{j}(t)\phi_{\alpha^{tot}_i}(y),\ \mathcal{I}^{hc}_{m,n}(\theta)(t,y)=\sum_{j=0}^{m-1}\sum_{i=1}^{|\Gamma_n(B)|}\theta_{j,i}q_{j}(t)\phi_{\alpha^{hc}_i}(y),\eeq
for all $(t,y)in \R\times\R^{d}$, and for $\theta$ in $\R^{m}\times\R^{|\Lambda_n(B)|}$ or $\R^{m}\times\R^{|\Gamma_n(B)|}$, respectively.\par
Following the previous notation, the total degree polynomial learning problem and the hyperbolic cross polynomial learning problem are respectively given by
\beq
\min_{\theta\in \R^{m}\times\R^{|\Lambda_n(B)|}}\mathcal{J}(\mathcal{I}^{tot}_{m,n}(\theta))+\gamma\left(\frac{1-r}{2}|\theta|^{2}_2+r|\theta|_1\right)
\label{PolynomialLearningProblem:total}
\eeq
and \beq
\min_{\theta\in \R^{m}\times\R^{|\Gamma_n(B)|}}\mathcal{J}(\mathcal{I}^{hc}_{m,n}(\theta))+\gamma\left(\frac{1-r}{2}|\theta|^{2}_2+r|\theta|_1\right)
\label{PolynomialLearningProblem:hc}
\eeq
For these problem we say that $m$ is the time degree and $n$ is the space degree.\par
Regarding the convergence and the existence of solutions of problems  \eqref{PolynomialLearningProblem:total} and \eqref{PolynomialLearningProblem:hc}, we known by Theorem 9 in \cite[Section 7.2]{Hayek} that hypothesis \eqref{Hyp:Seq:Dens} is satisfied as $m$ and $n$ tend to infinity. Hence, by  \Cref{Theo:convergence} there exist $(m_0^{tot},n_0^{tot})$ and $(m_0^{tot},n_0^{hc})$ such that problem \eqref{PolynomialLearningProblem:total} has an optimal solution if $m\geq m_0^{tot}$ and $n\geq n_0^{tot}$  and problem \eqref{PolynomialLearningProblem:hc} has an optimal solution if $m\geq m_0^{hc}$ and $n\geq n_0^{hc}$.
\section{Numerical Examples}
\label{NumericalExperimentsSection}
We implement \Cref{Alg1} to solve problem \eqref{AppLearningProblem} for three experiments. In all them $\mathcal{J}$ is approximated by means of the Monte-Carlo method. We choose uniformly at random a training set $\{(t_0^i,y_0^{i})\}_{i=1}^{N}$ of initial times and initial conditions and approximate $\mathcal{J}(v)$ by
$$ \mathcal{J}(v)\approx \frac{1}{N}\sum_{i=1}^{N}\mathcal{V}(t_0^i,y_0^i,v).$$
In each case we shall specify  the sets of initial conditions and initial times for training and testing, and the discretization method used. To avoid numerical issues we normalize the arguments of the monomials,  i.e., we redefine $\phi_{\alpha}(y)$ by
$\phi_{\alpha}(y)=\prod_{i=1}^{d}\left(\frac{y}{l}\right)^{\alpha_i}$ and $q_j^{m}(t)$ by $q_j^{m}(t)=\left(\frac{t}{T}\right)^{j}$, where $T$ is the time horizon of the control problem and $l$ is a positive constant chosen for each specific problem.
\par The performance of our approach is measured  by comparing $\hat{u}_i(t)=U_{\hat{v}}(t,\Phi(t_0^{i},y_0^{i},\hat{v})(t))$ to the corresponding  optimal control $u^*$  obtained by solving  the open loop problem,  for every initial condition in the test set given by $\hat{u}_i(t)=U_{\hat{v}}(t,\Phi(t_0^{i},y_0^{i},\hat{v})(t))$, and subsequently computing the normalized squared error in $L^{2}((0,T);\R^{m})$ given by
$$
MNSE_u(\{\hat{u}_{i}\}_{i=1}^{N},\{u_i^{*}\}_{i=1}^{N})=
\sum_{i=1}^{N}\int_{t_0^i}^{T}|\hat{u_i}-u_i^{*}|^{2}dt\Big/\sum_{i=1}^{N}\int_{t_0^i}^{T}|u_i^{*}|^{2}dt,
$$
and analogously the error $MNSE_y(\{\hat{y}_{i}\}_{i=1}^{N},\{y_i^{*}\}_{i=1}^{N})$ for the states.
We also compare the optimal value of the open loop problem with the  objective function of \eqref{ControlProblem} evaluated in $\hat{u}$ by computing the mean normalized absolute error
$$ MNAE_J(\{\hat{u}_{i}\}_{i=1}^{N},\{u_i^{*}\}_{i=1}^{N})=\sum_{i=1}^{N}
|J(t_0^{i},u_i^{*},y_0^{i})-J(t_0^{i},\hat{u}_i,y_0^i)|\Big/\sum_{i=1}^{N}J(t_0^{i},u_i^{*},y_0^{i}).$$
\par
The optimal control for the  open loop problem is computed by  a gradient descent algorithm with a backtracking line-search.

\subsection{Linearized inverted pendulum}
We consider the following nonautonomous linear-quadratic problem
\beq
\begin{array}{c}
    \dis \min_{u\in L^{2}((0,T),\R)}\frac{1}{2} \int_{t_0}^{T}|y(t)|^{2}dt+\frac{\beta}{2} \int_{t_0}^{T}|u(t)|^2 dt +\frac{\alpha}{2}|y(T)|^{2}\\ \ecart
    \dis s.t.\ y'(t)= A(t)y+Bu(t),\ y(0)=y_0,
\end{array}
\label{LQP}
\eeq
with
\beq A(t)=\left(\begin{array}{cccc}
     0 & 1 & 0&0\\
    \ecart 0 & \dis-\frac{F_r+M_c'(t)}{M_c(t)} & 0 & 0 \\
    \ecart\dis 0 & 0 & 0 & 1\\
    \ecart\dis -\frac{g_a}{L_p} & 0 & \dis\frac{g_a}{L_p}& 0
\end{array}\right) \mbox{ and } B=\left(\begin{array}{c}
     0  \\
     \ecart 1 \\
     \ecart 0 \\
     \ecart 0
\end{array}\right).\label{InvertedPendulumEq}\eeq
where $T=1$, $t_0\in (0,1)$, $F_r=1$, $g_a=9.8$, $L_p=0.842$, and
$M(t)=\exp(-10t)+1.$ Equation \eqref{InvertedPendulumEq} describes the linearized motion of an inverted pendulum mounted on a car.   Here, differently from  \cite{Kwakernaak}, we consider a time dependent mass.
\par
In this example we sample the initial conditions from $[-0.5,0.5]^{4}$ and the initial times from $[0,1]$. Due to the fact that the value function of this  linear-quadratic problem is a quadratic function of the initial condition we only solve \eqref{PolynomialLearningProblem:total} with the spatial degree equal to 2 and $l=0.5$. To study the convergence of this problem with respect to the cardinality of the training set we sample sets with sizes 8,12 ,16, and 20 as training set. For the test set we sample a set of cardinality 500. Since the dimensionality of this problem is small we do not investigate for this problem  sparsity of its solutions, and fix  $\gamma=10^{-5}$ and $r=0$. On the other hand, since we do not have information about the dependence of the value function on time, we consider 5, 10, 15, 20, 25 and 30 as time degrees. To discretize the closed and open loop problems we used an explicit Euler scheme. For solving the learning problem we used \Cref{Alg1} with $n=5$ and $tol=10^{-5}$.
\par
We see in \Cref{InvertedPendulum:trainerror} that for all the training sets the $MNSE_u$, the $MNSE_y$, and the $MNAE_J$ are decreasing with the time degree up to degree 20. Further, we observe the same for the test set in \Cref{InvertedPendulum:testerror} as the time degree and the cardinality of the time degree increase. We note that for the times degrees larger than 15 and training size bigger than  12, the test errors in \Cref{InvertedPendulum:testerror} do not change significantly.
\plotErrors{InvertedPendulumVarMass}{train}{Linearized inverted pendulum}{InvertedPendulum}
\plotErrors{InvertedPendulumVarMass}{test}{Linearized inverted pendulum}{InvertedPendulum}
\subsection{Allen-Cahn Equation}
We now consider the Allen-Cahn equation with the Neumann boundary conditions and quadratic cost, namely,
\beq
\begin{array}{c}
     \dis \min_{u_i\in L^{2}([0,T),\R)} \int_{t_0}^{T}\int_{-1}^{1}|y(x,t)|^{2}dxdt+\frac{\beta}{2}\int_{t_0}^{T}|u(t)|^{2}dt +\frac{\alpha}{2}\int_{-1}^{1}|y(T,x)|^2 dx\\
     \ecart\dis y'(t,x)=\nu\frac{\partial^{2}y}{\partial x^{2}}(t,x)+y(t,x)(1-y^2(t,x))+\sum_{i=1}^{3}\chi_{\omega_i}(x)u_i(t)\\
     \ecart\dis\frac{\partial y}{\partial x}(t,-1)=\frac{\partial y}{\partial x}(t,1)=0, \quad  y(0,x)=y_0(x),
     \label{AllenCahn}
\end{array}
\eeq
for $x\in (-1,1)$ and $t>0$, where  $\nu=0.5$, $T=4$, and $\chi_{\omega_i}$ are the indicator functions of the sets $\omega_{1}=(-0.7,-0.4)$, $\omega_{2}=(-0.2,0.2)$, and $\omega_{3}=(0.4,0.7)$. This problem admits 3 steady states, which are $-1$, $0$ and $1$, with $0$ being unstable.\par
 Since problem \eqref{AllenCahn} is infinite-dimensional, we discretize it by using a Chebyshev spectral collocation method with $39$ degrees of freedom. The first integral in \eqref{AllenCahn}  is approximated by means of the Clenshaw-Curtis quadrature. For further details on the Chebyshev spectral collocation method and the Clenshaw-Curtis quadrature we refer to \cite[Chapters 6, 19]{Boyd}, and \cite[Chapters 7, 12, 13]{Trefethen}.\par
 Due to the high dimensionality of this problem, the evaluation of the feedback law is computationally expensive. Therefore we only solve the hyperbolic cross polynomial learning problem \eqref{PolynomialLearningProblem:hc}. Further, the sparsity of the solution can be influenced by the penalty coefficient $\gamma$. With this in mind, we pay attention to the influence of  $\gamma$ on the sparsity of the solution and the performance of the obtained feedback laws.\par
 We sample the initial conditions for training and test from $[-10,10]^{39}$ and the initial times from $[0,1]$. The cardinality of the training and test sets are 10 and 500 respectively. For learning problem we set the time degree to 15, the sparsity coefficient to $r=0.5$ and $l=10$. For the learning algorithm we  set $n=5$ and $tol=10^{-5}$. We used the Crank–Nicolson scheme to discretize the closed and open loop problems.
\par
In \Cref{AllenCahn:trainerror} and \Cref{AllenCahn:testerror} the training and test errors are depicted. From \Cref{AllenCahn:trainmnae} and \Cref{AllenCahn:testmnae} we see that the MNAE$_J$ is lower for the spatial degree  4 than for the spatial degree  2,  as expected. Moreover, we note that the MNAE$_J$ decreases as $\gamma$ decreases for $\gamma>10^{-2}$ and for $\gamma\leq 10^{-2}$ the training MNAE$_J$ is lower than 2\% and the test MNAE$_J$ is lower than 6\%. The same behaviour is observed in \Cref{AllenCahn:trainmnsey} and \Cref{AllenCahn:testmnsey}  for the MNSE$_y$. Nevertheless, in \Cref{AllenCahn:trainmnseu} and \Cref{AllenCahn:testmnseu} we see that the MNSE$_u$ tend to increase as $\gamma$ decreases and the MNSE$_u$ is above 200\%. We point out that this does not contradict \Cref{Theo:convergence}, because it only ensures the convergence with respect to the value function. One possible  explanation for this is a possible lack of uniqueness of solutions of the open loop problem. Regarding the sparsity of the solutions of the learning problem we have to turn our attention to \Cref{AllenCahn:errortable}. In this table the cardinality of the support of the solutions of the learning problem for each space degree and $\gamma$ are shown along with the respective train and test MNAE$_J$. Here we define the support of $\theta\in \R^{m}$ as the set of indexes given by $\{i:\theta_i\neq 0\}$. We observe that the sparsity of the solutions increases with $\gamma$. On the other, the training and test MNAE$_J$  decrease with $\gamma$ for $\gamma \ge 10^{-4}$.

\begin{table}[h!]
\centering
\begin{tabular}{|c|c|c|c|c|c|}\hline
$\gamma$&\makecell{Spatial\\degree}&\makecell{Training\\MNAE$_{J}$ [\%]}&\makecell{Test\\MNAE$_{J}$ [\%]}&\makecell{Percentage\\ Support \\ cardinality [\%]}&\makecell{Support \\ cardinality}\\ \hline
1.0e-01&2&4.16&10.03&21.54&84\\
1.0e-01&4&3.27&9.72&1.91&134\\
\hline[0.1pt]
1.0e-02&2&2.33&5.97&85.13&332\\
1.0e-02&4&0.71&4.07&23.06&1619\\
\hline[0.1pt]
1.0e-03&2&2.12&5.38&97.44&380\\
1.0e-03&4&0.30&3.50&70.24&4931\\
\hline[0.1pt]
1.0e-04&2&1.78&4.56&100.00&390\\
1.0e-04&4&0.52&3.61&99.63&6994\\
\hline[0.1pt]
1.0e-05&2&1.92&4.80&100.00&390\\
1.0e-05&4&0.36&3.54&99.94&7016\\
\hline\end{tabular}
\caption{Cardinality and errors Allen-Cahn equation.}
\label{AllenCahn:errortable}
\end{table}

\plotErrors{ControlledAllenCahnProblem}{train}{Allen-Cahn equation}{AllenCahn}
\plotErrors{ControlledAllenCahnProblem}{test}{Allen-Cahn equation}{AllenCahn}

 \subsection{Collision-Avoiding multi-agent control problem}
 In this example we consider a set of $N_a\geq 2$ agents with states $\{(x^{a}_i)\}^{N_a}_{i=1}\subset \R^{p}$, a set of targets $\{(x^{t}_i)\}^{N_a}_{i=1}\subset \R^{p}$ to be reached by each agent and a set of $N_o\geq 1$ obstacles denoted by $\{(x^{o}_i)\}^{N_a}_{i=1}\subset \R^{p}$. The control is on the velocity of the agents, i.e.,
 $ (x_i^a)'(t)=u_a(t)$  for $t\in (t_0,T)$. The running cost  is given by
 $$ \ell(x_1^{a},\ldots,x_{N_a}^{a})=\sigma_1Q(x_1^{a},\ldots,x_{N_a}^{a})+\sigma_2W(x_1^{a},\ldots,x_{N_a}^{a}) $$
 where
 $$Q(x_1^{a},\ldots,x_{N_a}^{a})=\frac{1}{N_a\cdot N_o}\sum_{i=1}^{N_a}\sum_{j=1}^{N_o}\exp(-\frac{1}{2r_o^2}|x_i^a-x_j^o|^2)\frac{1}{r_o^q}$$
 and
 $$ W(x_1^{a},\ldots,x_{N_a}^{a})=\frac{2}{N_a\cdot(N_a-1)}\sum_{i=1}^{N_a}\sum_{j=i+1}^{N_a}\exp(-\frac{1}{2r_a^2}|x_i^a-x_j^a|^2)$$
 with $r_o,r_a>0$. The term $Q$ accounts for collision between agents and obstacles, and $W$ accounts for the collision among agents. To enforce the agents to reach the targets, the following terminal cost is used
 $$ g(x_1^{a},\ldots,x_{N_a}^{a})=\frac{\sigma_3}{2N_a}\sum_{i=1}^{N_a}|x_i^a-x_i^t|^{2}.$$
 The resulting optimal control problem is given by
 \begin{equation}
     \begin{array}{c}
         \dis \min_{ u_i\in L^{2}((t_0,T);\R^{p})}\int_{t_0}^{T}\ell(x_1^{a}(t),\ldots,x_{N_a}^{a}(t)) dt+\sum_{i=1}^{N_a}\frac{\beta}{2N_a}\int_{t_0}^{T}|u_i(t)|^{2}dt+g(x_1^{a}(T),\ldots,x_{N_a}^{a}(T)),\\
          \ecart\dis \frac{dx_i^{a}}{dt}(t)=u_a(t), \quad t\in (t_0,T),\quad x_a^i(0)=x_{a,0}^i,\quad i\in\{1,\ldots,N_a\}.
     \end{array}
 \end{equation}
This formulation is similar to those given in \cite{Onken} and \cite{Onken2}.\par
For our implementation  we choose $p=2$, $N_a=10$, $N_o=4$, $r_o=0.2$, $r_a=0.1$, $\sigma_1=10$, $\sigma_2=10$ and $\sigma_3=100$. The obstacles and targets are given by $x_i^{o}=\frac{1}{2}\left(\cos\left(\frac{2(i-1)\pi}{N_o}\right),\sin\left(\frac{2(i-1)\pi}{N_o}\right)\right)$ for i in $\{1,\ldots,N_o\ldots\}$  and $x_i^{t}=\frac{1}{2}\left(\cos\left(\frac{2(i-1)\pi}{N_a}\right),\sin\left(\frac{2(i-1)\pi}{N_a}\right)\right)$ for i in $\{1,\ldots,N_a\}$. Both the targets and the obstacles are depicted in \Cref{ObstacleProb:ObsTar}.
\par
For the learning problem we consider $\omega=\omega_1\times\omega_2\ldots \times \omega_{N_a}$, where $$\omega_i=\left\{(\rho\cos(\theta),\rho\sin(\theta))\in \R^{2}:\ 0.8\leq \rho\leq 2,\ \theta\in \left(\frac{2\pi}{N_a}(i-1),\frac{2\pi}{N_a}i\right)\right\}.$$
The cardinality of the training and test sets are 10 and 500 respectively. For the learning problem we set the time degree to 15, $l=2$, and the sparsity coefficient to $r=0.5$. For the learning algorithm we  set $n=5$ and $tol=10^{-4}$. We used the explicit Euler scheme to discretize the closed and open loop problems.
\par
In \Cref{ObstacleProb:trainerror} and \Cref{ObstacleProb:testerror} the training and test errors are depicted. From \Cref{ObstacleProb:trainmnae} and \Cref{ObstacleProb:testmnae} we see that all the errors are lower for the space degree equal to 4 than for the case when the space degree is equal to 2 as expected. Moreover, we note that for $\gamma\leq 10^{-2}$ the train and test MNAE$_J$ are lower than 5\%. The cardinality of the support of the solutions of the learning problem for different spatial degrees and $\gamma$ are shown in \Cref{ObstacleProb:errortable}. As in the Allen-Cahn problem, we observe that the sparsity of the solutions increases with $\gamma$, but the train and test MNAE$_J$ increase for $\gamma$ sufficiently large.

\begin{table}[h!]
\centering
\begin{tabular}{|c|c|c|c|c|c|}\hline
$\gamma$&\makecell{Spatial\\degree}&\makecell{Training\\ MNAE$_{J}$ [\%]}&\makecell{Test\\ MNAE$_{J}$ [\%]}&\makecell{Percentage\\ Support \\ cardinality}&\makecell{Support \\ cardinality}\\ \hline
1.0e-01&2&25.46&29.98&66.50&399\\
1.0e-01&4&16.38&26.90&30.77&1246\\ \hline[0.1pt]
1.0e-02&2&7.17&8.23&90.83&545\\
1.0e-02&4&5.36&7.11&38.07&1542\\ \hline[0.1pt]
1.0e-03&2&6.65&7.37&100.00&600\\
1.0e-03&4&3.01&5.74&91.11&3690\\ \hline[0.1pt]
1.0e-04&2&4.29&4.98&100.00&600\\
1.0e-04&4&2.76&4.53&99.85&4044\\ \hline[0.1pt]
1.0e-05&2&4.64&5.31&100.00&600\\
1.0e-05&4&2.93&4.71&99.95&4048\\
\hline\end{tabular}\label{ObstacleProb:errortable}
\caption{Cardinality and error collision-avoiding multi-agent control problem.}
\end{table}

\begin{figure}[h!]
\center
    \includegraphics[width=0.5\textwidth]{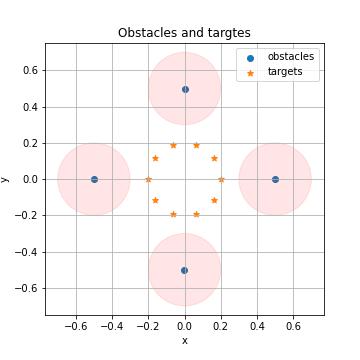}
\caption{Obstacles and targets.}
\label{ObstacleProb:ObsTar}
\end{figure}
\plotErrors{ObstacleProblem}{train}{Collision-Avoiding multi-agent control problem}{ObstacleProb}
\plotErrors{ObstacleProblem}{test}{Collision-Avoiding multi-agent control problem}{ObstacleProb}

\par
\section{Conclusion}
In this work a learning approach based on a polynomial ansatz for the synthesis of feedback laws of finite horizon control problem was studied. The performance of this approach was tested on 3 different problems. In all the experiments, the feedback law synthesised by the presented approach was capable of approximating the solutions of the corresponding open loop problems. Further, our experiments show that by tuning the penalty coefficients and the degree of the polynomial ansatz it is possible to find sparse polynomial feedback laws.

\vskip 0.2in
\bibliography{biblio}

\end{document}